\documentclass{siamart0216}
\usepackage{amssymb}
\usepackage[title]{appendix}
\usepackage{wrapfig}
\usepackage{hyperref}
\newtheorem*{thm_restated}{Theorem}

\newcommand{\norm}[1]{\left\Vert#1\right\Vert}

\newcommand{\brac}[1]{\left [#1\right ]}
\newcommand{\tr}[1]{\mathrm{tr} #1}

\newcommand{\Real}{\mathbb R}

\newcommand{\A}{\mathcal{A}}
\newcommand{\OO}{\mathcal{O}}

\newcommand {\1}{\mathrm{\textbf{1}}}
\newcommand{\one}{\1}

\def \diag{\mathrm{diag}}

\def \M{\mathcal{M}} 
\def \PP{\mathcal{P}}
\def \S{\mathcal{S}}

\newcommand{\J}{\mathcal{J}}
\newcommand{\dist}{\mathrm{d}} 

\newcommand{\fnorm}[1]{\norm{#1}_F}

\DeclareMathOperator{\conv}{conv}

\newcommand{\eg}{{\it e.g.}}
\newcommand{\ie}{{\it i.e.}}

\newcommand{\perm}{\Pi}
\newcommand{\dConv}{\underline{\dist}} 

\renewcommand {\vec}[1]{\mathbf{#1}}

\newcommand{\RR}{\mathbb{R}}
\newcommand{\MF}{\mathcal{M}_F} 
\renewcommand{\M}{\mathcal{M}}
\newcommand{\N}{\mathcal{N}}

\newcommand{\Rmu}{\mu[R]}

\newcommand{\Iso}{\mathrm{ISO}}

\newcommand{\bigO}{O}
\newcommand{\Prob}{\mathbb{P}}

\newcommand{\twopartdef}[4]
{
        \left\{
                \begin{array}{ll}
                        #1 & \mbox{if } #2 \\
                        #3 & \mbox{if } #4
                \end{array}
        \right.
}

\title{Exact Recovery with Symmetries\\ for Procrustes Matching}
\author{Nadav Dym and Yaron Lipman}
\begin{document}

\sloppy
\maketitle

        \begin{abstract}
                {The Procrustes matching (PM) problem is the problem of finding the optimal rigid motion and labeling of two point sets so that they are as close as possible. Both rigid and non-rigid shape matching problems can be formulated as PM problems. Recently \cite{Haggai} presented a novel convex semi-definite programming relaxation (PM-SDP) for PM which achieves state of the art results on common shape matching benchmarks.

                In this paper we analyze the successfulness of PM-SDP in solving PM problems without noise (Exact PM problems). We begin by showing Exact PM to be computationally equivalent to the graph isomorphism problem. We demonstrate some natural theoretical properties of the relaxation, and use these properties together with the moment interpretation of \cite{Lasserre} to show that for exact PM problems and for (generic) input shapes which are asymmetric or bilaterally symmetric, the relaxation returns a correct solution of PM.

                For  symmetric shapes, PM has multiple solutions. The non-convex set of optimal solutions of PM is  strictly contained in the convex set of optimal solutions of PM-SDP, so that `most' solutions of PM-SDP will not be solutions of PM. We deal with this  by showing the solution set of PM to be the extreme points of the solution set of PM-SDP, and suggesting a random algorithm which returns a solution of PM with probability one, and returns all solutions of PM with equal probability. We also show these results can be extended to the almost-exact case.
                 To the best of our knowledge, our work is the first to achieve exact recovery in the presence of multiple solutions.
                         }
        \end{abstract}



\section{Introduction}\label{sec:intro}

Shape comparison is a central task in many fields such as computer graphics, computer vision, medical imaging and biology in general. The input of the problem is a pair of shapes, often represented by $d$-dimensional  point clouds  $P\in\Real^{d\times n}$, $Q\in \Real^{d\times n}$, where a point is represented by the  column of a matrix. The goal is to determine how similar the shapes are to each other, typically by computing a distance between shapes which is invariant to a prescribed class of shape preserving deformations.

In rigid shape matching problems, shape preserving deformations are rigid motions, as well as permutations of the point cloud which change the order the points are given but not the shape defined by them. A distance between  shapes for rigid problems can be defined by finding a linear isometry $R \in \OO(d) $, a translation $t \in \RR^d $ and a permutation $X \in \perm_n $ minimizing the Euclidian distance between the point clouds. Centralizing the point clouds causes the optimal translation to be $t=0 $, and we arrive at the following optimization problem:

\begin{subequations} \label{e:procrustes}
        \begin{align}
                \quad \dist^{2}(P,Q)=\min_{X,R} &  \fnorm{RP-QX}^2   \\
                \mathrm{s.t.} & \quad X \in \perm_n \\
                & \quad R \in \OO(d)
        \end{align}
\end{subequations}
We will refer to this optimization problem as the Procrustes matching (PM) problem. We note that both the minimal value and the minimizers are of practical importance. The minimal value is a measure of the similarity of the shapes, while the minimizers $R,X $ enable correct alignment and labeling of the shapes.

For non-rigid shapes, shape preserving deformations  also include intrinsic non-rigid isometries. In computer graphics, a popular strategy   for tackling the non-rigid problem is embedding the shapes in a higher dimensional space, where non-rigid isometries of the shapes are approximated by linear isometries in the high dimensional space (\eg, \cite{ovsjanikov2008global}) . As a result the non-rigid problem can also be formulated as PM, though now $d$ is high dimensional as opposed to the rigid case where typically $d=3 $.

\begin{figure}[t]
        \centering
        \includegraphics[width=0.85\columnwidth]{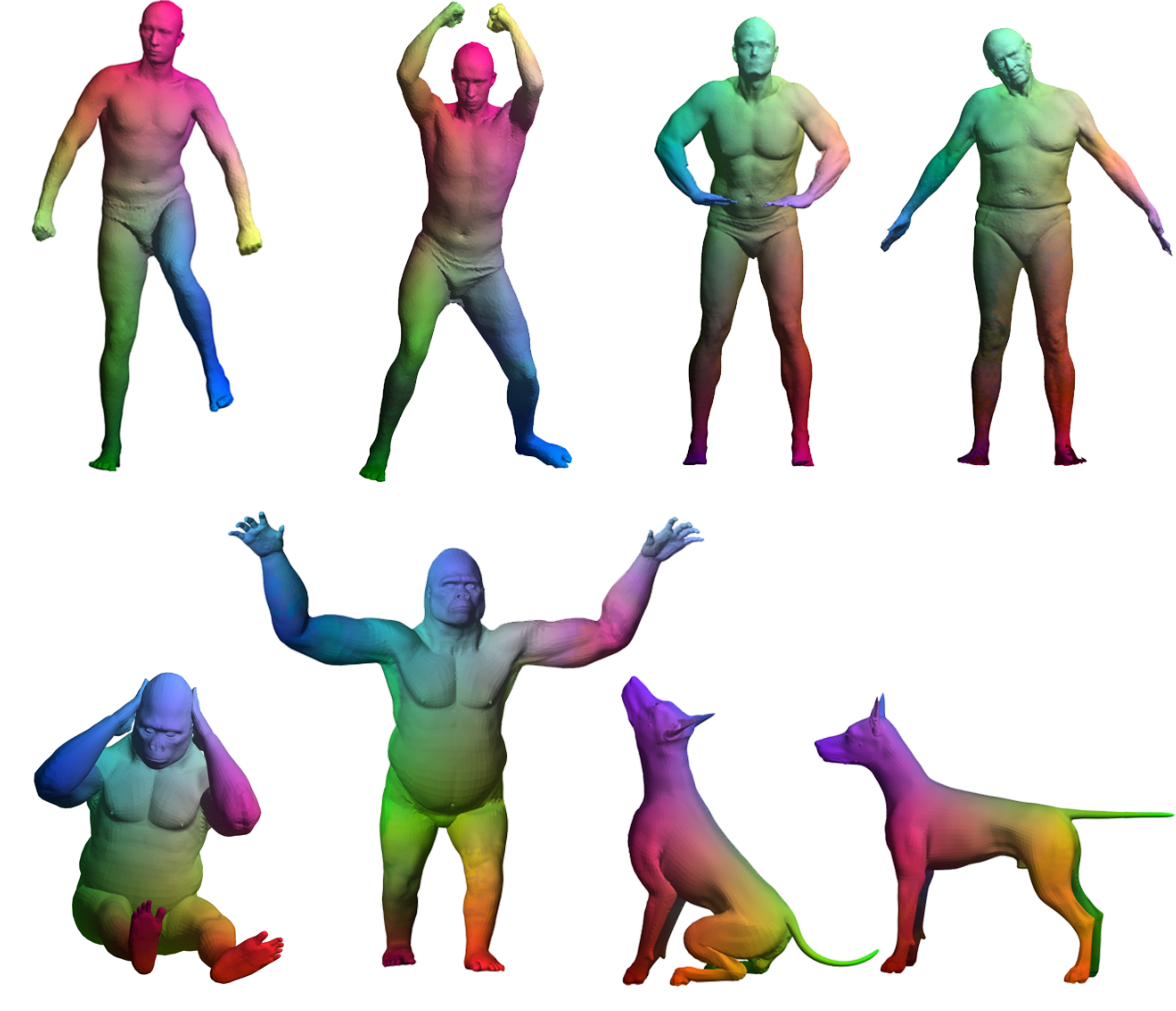}
        \vspace{+0cm}
        \caption{Examples of non-rigid matching results obtained by PM-SDP. Results are visualized by transforming a color map on the source shape (left) to the target shape (right) according to the correspondence found by PM-SDP. }
        \label{fig:hm}
        \vspace{-0.1 cm}
\end{figure}

PM is a non-convex optimization problem and global minimization is difficult. In fact we show that even the subproblem of exact PM is difficult.  Exact PM refers to the problem of deciding whether $\dist(P,Q)=0 $ or not, and the related  problem  of finding minimizers $(R,X)$ to \eqref{e:procrustes} when the Procrustes distance is zero. We show this by proving \vspace{3pt}
\begin{theorem}\label{th:GM}
There is a polynomial reduction from exact GM to exact PM, and vice versa. \vspace{3pt}
\end{theorem}
Exact GM is a well studied problem which is not yet known to be either polynomial or NP-complete, although recently it has been shown to be solvable in quasi-polynomial time \cite{DBLP:journals/corr/Babai15}.
In \cite{Haggai} we presented a method for approximating the solution of PM by formulating a semi-definite convex relaxation which we name PM-SDP. The relaxation is constructed using standard methods for semi-definite relaxations of quadratic problems. Using results on semi-definite matrix completion,  the relaxation is reduced to a considerably more efficient SDP relaxation which is equivalent to the original relaxation. The usefulness of PM-SDP for rigid and non-rigid matching problems was demonstrated as well. Some examples of non-rigid shape matching using PM-SDP are presented in Figure~\ref{fig:hm}.

Our main goal in this paper is to theoretically justify the successfulness of PM-SDP on exact or near exact PM problems.
 This successfulness was demonstrated empirically by the non-rigid results in \cite{Haggai} and by experiments presented there which showed that for exact and almost-exact problems the correct solution of PM was retrieved. We show that while solving exact PM is computationally hard in general, under assumptions which are typically valid for  shape matching problems, PM-SDP is guaranteed to correctly solve exact and near-exact PM problems. In particular our assumptions are valid for `most' asymmetric or bilaterally symmetric shapes. The latter class includes many important shape matching instances such as humans and other animals; for example, all the shapes in Figure~\ref{fig:hm} are bilaterally symmetric. For simplicity of the exposition we currently refer to $P$ which satisfy our assumptions as `generic', and  explain our assumptions afterwards.

Apart from the exactness results, we also show some natural properties of  the relaxation. We show invariance of the algorithm to orthogonal transformation and relabeling of the input,  that the $R,X $ coordinates of the solution are always in the convex hull of $\OO(d) \times \perm_n $, and that the objective value $\dConv(P,Q) $ of PM-SDP is always non-negative, and is never larger that $\dist(P,Q)$.

\subsection{Main results}

From the properties of the relaxation we just mentioned it follows that for exact problems

$$0 \leq \dConv(P,Q) \leq \dist(P,Q)=0$$

so that PM-SDP attains the correct objective value for exact problems. Our goal is to show that PM-SDP attains the correct minimizers as well.

For exact problems, we will refer to minimizers of PM as \emph{exact solutions} and will call minimizers of PM-SDP \emph{convex exact solutions}. Note that the set of \emph{convex exact solutions} is a convex set containing the set of exact solutions.  For generic asymmetric point-clouds we show that these two sets are equal: \vspace{3pt}

\begin{theorem}[asymmetric version] \label{th:main}
        Assume $P$ is generic and asymmetric, and $\dist(P,Q)=0$. Then PM-SDP has a unique convex exact solution, which coincides with the unique exact solution of $PM$. \vspace{3pt} 
\end{theorem}

For symmetric $P$,$Q$, there are several exact solutions. Each such solution is also a convex exact solution. In fact, all convex combinations of these solutions will be convex exact solutions as well, so that in the symmetric case convex exact solutions will not necessarily be exact solutions. To deal with this problem we will need some notation:

We denote the set of exact  solutions by $\Iso(P,Q)$
and the projection of this set onto the $R$ coordinate by $\Iso_R(P,Q)$. We refer to members of this set as exact orthogonal solutions. Similarly we denote the set of convex exact solutions by $\N(P,Q)$ and its projection onto the $R$ coordinate by $\N_R(P,Q)$. We refer to the members of this set as convex exact orthogonal solutions. We prove that the extreme points of the set of convex exact orthogonal solutions is exactly the convex hull of the set of exact orthogonal solutions: \vspace{3pt}

    \begin{thm_restated}[Theorem~\ref{th:main} \textup{(symmetric version)}]
        Assume $P$ is generic and $\dist(P,Q)=0$. Then
        $$ \Iso_R= \mathrm{ext}(\N_R) $$
        where $\mathrm{ext}(\N_R)$ denotes the extreme points of $\N_R$. \vspace{3pt}
    \end{thm_restated}

Under the condition of the theorem, we obtain a semi-definite characterization of the convex hull of the optimal set of PM, as the intersection of the feasible set of PM-SDP with the hyperplane defined by constraining the objective to be zero. We can then obtain extreme points (=exact orthogonal solutions) by a random algorithm which maximizes random linear energies over this set. This strategy for obtaining extreme points was suggested by Barvinok in \cite{barvinok1995problems}. We prove that if the linear energies are selected according to the  uniform distribution on the unit sphere, then
\begin{theorem}\label{th:proj}
        The random algorithm returns an exact solution with probability one. Moreover the probability distribution induced on the exact solutions is uniform.
\end{theorem}
We also show the random algorithm can be adapted to the almost exact case, and maintain similar theoretical guarantees.

An experiment illustrating our main results is shown in Figure~\ref{fig:sym}.
\begin{figure}[t]
        \centering
        \includegraphics[width=\columnwidth]{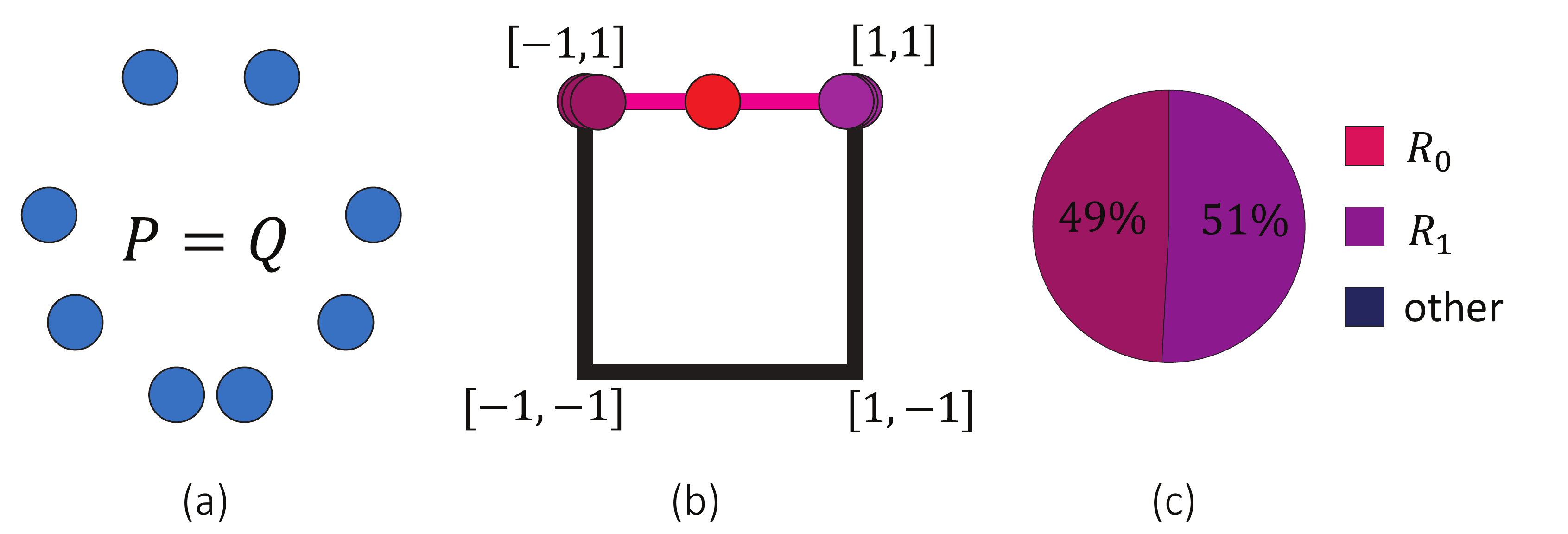}
        \vspace{-0cm}
\caption{An experiment illustrating our main result. We apply PM-SDP to bilaterally symmetric point clouds $P=Q$ seen in (a). PM has two solutions $(R_i,X_i), i=0,1$, where $R_i$ are diagonal matrices with values  $[1,1] $ or $[-1,1]$ on the diagonal. The solution of PM-SDP (the red dot in (b) ) is a convex combination of $R_0 $ and $R_1$ as predicted by Theorem~\ref{th:main}. By applying a random projection the correct solutions of PM (purple dots in (b)) are obtained with equal probability as seen in the pie chart in (c). This is in accordance with Theorem~\ref{th:proj}.}
        \vspace{-0.3 cm}
        \label{fig:sym}
\end{figure}

\paragraph{Related work}
A popular strategy for local minimization of PM is the ICP algorithm \cite{besl1992method,Rusinkiewicz2001} which iteratively solves for $R$ and then $X$, while holding the other coordinate fixed.  ICP enjoys excellent scalability, but requires good initialization to avoid local minima. Another popular strategy is RANSAC \cite{fischler1981random} which repeatedly randomly selects $d+1$ points from each shape and finds the optimal rigid motion between them. The global rigid motion is then found or approximated by applying a voting scheme to a sufficient number of trials. For high dimensional data this method is problematic since its complexity is exponential in $d$.

\cite{DBLP:journals/corr/KhooK15} propose a natural convex relaxation to PM by optimizing over the convex hull of the feasible set $\OO(d) \times \perm_n$. This relaxation is faster than PM-SDP but is provably less accurate (see \cite{Itay} for a similar argument for the graph matching problem). A drawback of this relaxation is that while exact recovery is guaranteed for certain asymmetric shapes, it is also known to fail for centralized point clouds since in this case $R=0,X=\frac{1}{n}\one \one^T$ is always a feasible zero-objective solution.

There are many examples of convex relaxations of non-convex problems for which exact recovery results are available, \eg, \cite{so2007theory,Huang:2013:CSM:2600289.2600314,aflalo2015convex}, often under certain assumptions on the data. The papers mentioned above all address situations in which the non-convex problem has a unique solution and show that the relaxation has a unique solution as well. To the best of our knowledge our treatment here is the first proof of exact recovery for a relaxation of symmetric problems which possess multiple solutions.

\subsection{Conditions} We now explain our assumptions on the point cloud $P$. A weak assumption we use is that $P$ does not include the same point twice. Our main assumption is  that the spectrum of the matrix $PP^T$ is simple (all eigenvalues have unit multiplicity). This assumption implies that exact orthogonal symmetries are compositions of reflections with respect to the principal axes of $P$. In other words, in the principal axes basis an orthogonal symmetry is diagonal, with diagonal entries in $\{-1,1\}$. In particular all symmetries are bilateral ($R^2=I_d$). We say the $R$ is a member of  $\{-1,1\}^d$ if it is such a diagonal matrix.

The simple spectrum assumption and the weaker bounded eigenvalue multiplicity assumption are common assumptions in the exact graph matching literature \cite{bounded_eig,aflalo2015convex}. It is known that in this case exact GM can be solved by an algorithm which is polynomial  in the input size but exponential in the eigenvalue multiplicity. This suggests that exact recovery for PM-SDP should fail at some level of eigenvalue multiplicity. Indeed we provide  examples where such failures occur even for multiplicity two.

Our third and final assumption is the faithfulness assumption. We say a point $P_j$ is \emph{faithful}, if for any  $R \in \{-1,1\}^d $ for which $RP_j $ is a point of $P$, $R$ is an orthogonal symmetry. We say that $P$ is faithful if it contains at least one faithful point.

An illustration of  the faithfulness condition is given in the inset.  The image on the left of the inset shows a shape from the Scape dataset \cite{scape} which when sampled has a faithful point marked in green.
\begin{wrapfigure}{r}{0.30\textwidth}
\vspace{0 cm}\hspace{-0.2 cm}
     \includegraphics[width=0.30\textwidth]{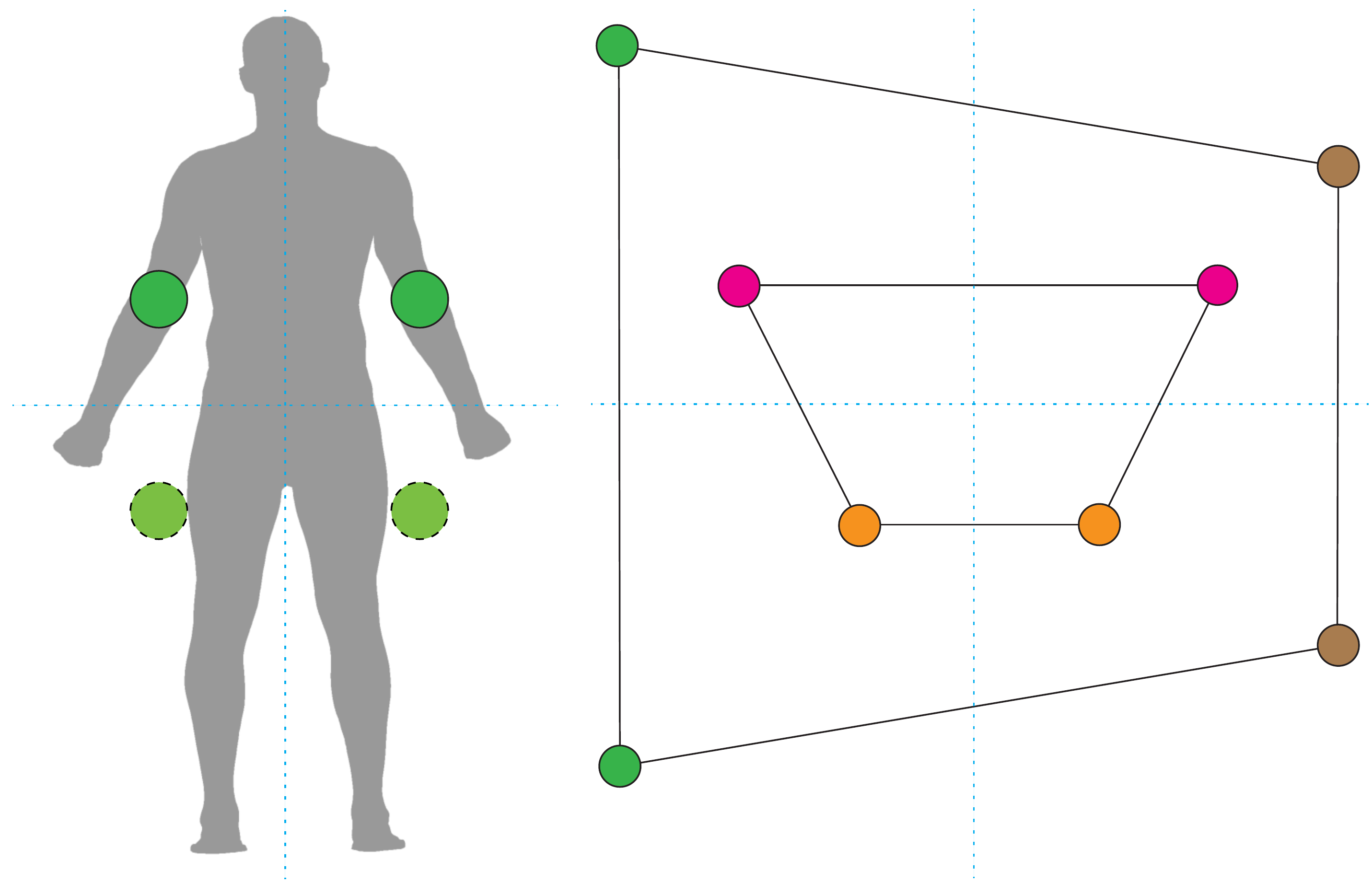}
  \vspace{-0.2 cm}
\end{wrapfigure}
Applying the reflectional symmetry $R \in \Iso_R(P,P) $ to the point in green gives another green point on the shape, while applying $R \in \{-1,1 \}^2 $ which aren't  symmetries gives points which aren't on the shape (dotted, light green).  The image on the right shows a  synthetic example of a point-cloud where all points are unfaithful, since the point-cloud is asymmetric but each point can be reflected to a different point in the point cloud.

Our experiments show that PM-SDP is able to correctly solve the unfaithful exact PM problem displayed in the inset, as well as similar randomly generated problems. Thus, in contrast to the simple spectrum condition which is tight, we do not currently know whether the faithfulness condition can be removed or replaced by a weaker condition.

\subsection{Paper organization}
We begin by presenting our SDP relaxation for PM in Section~\ref{sec:relaxation}.  We prove the theoretical properties of the relaxation in Section~\ref{sec:properties} and use them in Section~\ref{sec:exactRecovery} to prove Theorem~\ref{th:main}. In Section~\ref{sec:extreme} we describe the random algorithm for obtaining extreme points and prove Theorem~\ref{th:proj}. We then explain how this result can be extended to accommodate near exact problems. In Section~\ref{sec:exp} we present some experiments which illustrate our results, and examine the behavior of PM-SDP for exact problems which don't meet the conditions of our analysis. Finally we  show the equivalence of exact graph matching and exact PM in Section~\ref{sec:complexity}.

\subsection{Notation}
We denote the set $\{1,\ldots,n \}$ by $[n]$.

We denote polynomials of degree $\leq r $ in $x=(x_1,\ldots,x_{N}) $ by $\PP_r(x)$.

We denote the feasible set of \eqref{e:procrustes} by
$G=\OO(d) \times \perm_n $.

We denote the $j$-th column of a matrix $X \in \RR^{n \times n} $ by $X_j\in \RR^{n \times 1}$,  and the $i$-th row by $X_{i \star} \in \RR^{1 \times n}$. Expressions such as $X_j^T $ should be interpreted as $(X_j)^T $ (as opposed to $(X^T)_j$).

We denote by $\1$ the vector $\1=(1,1,\ldots,1)^T\in \RR^{n \times 1}$. All vectors are column vectors unless stated otherwise.

Real symmetric  $n\times n$ matrices are denoted by $\S(n)$. We use $A\succeq 0$ to denote positive semi-definite matrices.

\section{Convex relaxation}\label{sec:relaxation}
PM can be formulated using quadratic polynomials in the entries of $R,X$. For generality of the exposition, as well as for notational convenience, we will now review  a common strategy for semi-definite relaxation of quadratic optimization problems (see \cite{luo2010semidefinite} for a review) and will then state the formulation of the PM-SDP relaxation from \cite{Haggai}. We present the relaxation using the  terminology of \cite{Lasserre}, which we find helpful for noting properties of the relaxation which are otherwise less apparent.

\subsection{SDP relaxations for quadratic optimization problems}\label{sub:quadratic}
We consider quadratic optimization problems of the form
\begin{subequations}\label{e:general_quadratic}
        \begin{align}
        \min_{x \in \RR^N} & \quad f_0(x)   \\
        \mathrm{s.t.} & \quad f_s(x)=0, \quad  s=1, \ldots, S \label{e:quad_eq}
        \end{align}
\end{subequations}
where $f_s$, $s=0,\ldots, S$ are quadratic polynomials. We denote the set of points $x$ satisfying \eqref{e:quad_eq} by $K$. We will assume $K$ is compact. This non-convex optimization problem can be rephrased as an equivalent convex problem over the infinite dimensional space $\Prob(K)$ of probability measures supported in $K$:
\begin{align}\label{e:infinite}
\min_{\mu \in \Prob(K)} \mu[f_0]=\int f_0 d\mu
\end{align}
The equivalence is in the sense that the optimal values of both problems are equal, and  each solution $x$ of \eqref{e:general_quadratic} defines a solution $\delta_x$ of \eqref{e:infinite}. Solutions of \eqref{e:infinite} can also be (possibly continuous) convex combinations of such solutions $\delta_x$.
We use this equivalence to construct a finite dimensional convex relaxation of \eqref{e:general_quadratic}. Instead of considering $\Prob(K)$ we will consider a superset $\M$ consisting of (not necessarily positive) linear functionals on $C(K)$. A member $\mu\in\M$ must fulfill the following conditions:
\begin{enumerate}
\item $\mu[1]=1$
\item For all $s$, $\mu[f_s]=0 $, for $s=1,\ldots , S$.
\item For all $p \in \PP_1(x)$, $\mu[p^2]\geq 0 $. If $\mu$ satisfies this property we write $\mu \succeq 0 $.
\end{enumerate}
Note that indeed $\Prob(K)\subset \M$ since every probability measure on $K$ satisfies these conditions; in fact it satisfies the last condition for every non-negative integrable function.
Applying this relaxation to \eqref{e:infinite} gives
\begin{subequations}\label{e:general_relaxation}
        \begin{align} \label{e:general_relaxation_1}
        \min_{\mu \in \M} & \quad \mu[f_0]   \\ \label{e:general_relaxation_2}
        \mathrm{s.t.} & \quad \mu[1]=1\\ \label{e:general_relaxation_3}
        & \quad \mu[f_s]=0, \quad  s=1 \ldots S \\ \label{e:general_relaxation_sdp}
        & \quad \mu \succeq 0
        \end{align}
\end{subequations}
To see this is indeed a finite dimensional semi-definite program note that in \eqref{e:general_relaxation} $\mu$ is only applied to polynomials in $\PP_2(x)$ and therefore the unknowns of the problem are \begin{equation} \label{e:vars}
\left(\mu[1],\mu[x_1],\ldots,\mu[x_N],\mu[x_1^2],\mu[x_1x_2],\ldots,\mu[x_N^2] \right)
\end{equation}
Equations \eqref{e:general_relaxation_1}-\eqref{e:general_relaxation_3} are all linear  in the unknowns, while \eqref{e:general_relaxation_sdp} is a semi-definite constraint as explained next. For $p,q \in \PP_1(x) $
$$p(x)=\vec{p}^T \begin{pmatrix}1 \\
x \\
\end{pmatrix} \quad,
\quad
q(x)=\vec{q}^T \begin{pmatrix}1 \\
x \\
\end{pmatrix}  $$
where $\vec{p},\vec{q} $ are the coefficient vectors of $p,q$.
Then
\begin{align}\label{e:fun2mat}
\mu[pq]=\mu\left[\vec{p}^T \begin{pmatrix}1 \\
x \\
\end{pmatrix}
 \begin{pmatrix}1
 \; x^T \\
\end{pmatrix}
\vec{q}
 \right]
 =
 \vec{p}^T\mu\left[ \begin{array}{cc}
1 & x^T \\
x & xx^T
\end{array}
 \right]\vec{q} \nonumber  \end{align}
Where for a function $F:\RR^{N} \to \RR^{n \times k}$ such that $F_{ij}(x)\in \PP_2(x)$ we define $\mu[F]$ by applying $\mu$ to each coordinate, \ie, $\mu[F]_{ij}=\mu[F_{ij}]$.
It follows from \eqref{e:fun2mat} that
$$\mu \succeq 0 \text{ if and only if } \mu\left[ \begin{array}{cc}
1 & x^T \\
x & xx^T
\end{array}
\right] \succeq 0  $$
We denote the set of feasible solutions of \eqref{e:general_relaxation} by $\M_F$.
\subsection{Properties} We present some simple but important consequences of the discussion above:
\begin{proposition}\label{p:lasserre}
Assume $\mu \in \M_F$ and $p \in \PP_1(x) $,
\begin{enumerate}
\item If $\mu[p^2]=0$ then $\mu[pg]=0$ for all $g\in\PP_1(x)$.
\item  $ \mu[p^2] \geq \mu[p]^2$.
\end{enumerate}
\end{proposition}
\begin{proof}
For any positive $c$,
$$\mu\left[(c^{-1}p\pm cg)^2 \right]\geq 0 $$
which implies that
\begin{equation*}
2|\mu[pg]|\leq c^{-2}\mu[p^2]+c^2\mu[g^2]
\end{equation*}
The first claim follows by taking $c \rightarrow 0 $. The second claim follows by choosing $c=1$ and taking $g $ to be the constant function $g=\sqrt{\mu[p^2]} $.
\end{proof}

An immediate consequence is that if $p_1,\ldots,p_k \in \PP_1(x) $ all satisfy $\mu[p_i^2]=0 $, then $\mu[f]=0$ for all members $f$ of the vector space
$$\langle p_1,\ldots,p_k \rangle_2=\{f|f=\sum_{i=1}^k p_ig_i, \quad g_i \in \PP_1(x) \} $$
\subsection{Convex relaxation of the Procrustes problem} The PM problem is equivalent to the quadratic optimization problem:
\begin{subequations}\label{e:PM_full_quadratic_formulation}
        \begin{align}
        \min_{X,R} & \quad \fnorm{RP-QX}^2   \\
        \mathrm{s.t.}
         & \quad X \one =\one \ \ ,  \ \  \one^TX=\one^T  \\
        &\quad X_jX_j^T=\diag \left( X_j \right), \quad j=1\ldots n   \label{e:quadDiag}\\
        &\quad RR^T=I_d\ \ ,  \ \   R^TR=I_d
        \end{align}
\end{subequations}
where for a vector $v$, $\diag (v)$ is the diagonal matrix whose diagonal entries are the entries of $v$. We note that \eqref{e:quadDiag} follows from the fact that the entries of each column $X_j $ are in $\{0,1 \} $, and each column has only one non-zero entry. The equivalence of PM with \eqref{e:PM_full_quadratic_formulation} is explained in full detail in  \cite{Haggai}.

Applying the semi-definite relaxation described in the previous section to the quadratic formulation of PM we obtain our formulation for PM-SDP:
\begin{subequations}\label{e:PM_SDP}
        \begin{align}
        \dConv^{2}(P,Q)=\min_{\mu} & \quad \mu\brac{\fnorm{RP-QX}^2}  \label{e:obj} \\
        \mathrm{s.t.}
         & \quad \mu[1]=1\\
         & \quad \mu[X] \one =\one \ \ ,  \ \  \one^T \mu[X]=\one^T  \label{e:X1}\\
        &\quad \mu[X_j X_j^T]=\diag\left( \mu[X_j] \right), \quad  j=1\ldots n \label{e:diag} \\
        &\quad \mu[RR^T]=I_d\ \ ,  \ \   \mu[R^TR]=I_d \label{e:orth} \\
        &\quad \mu \succeq 0 \label{e:PMSDP_psd}
        \end{align}
\end{subequations}

 An important property of PM-SDP (which is not needed for our discussion here) is that \eqref{e:PMSDP_psd} can be replaced with the semi-definite constraints
\begin{equation}\label{e:small_psd}
 \mu\left[ \begin{array}{cc}
 1 & x_j^T \\
 x_j & x_jx_j^T
 \end{array}
 \right] \succeq 0  , \quad j=1,\ldots,n
\end{equation}
to obtain an \emph{equivalent} relaxation. Here $x_j$ is a vector containing the entries of $R$ and the column  $X_j$. In applications often $n>>d $, and in this case the $\bigO(n+d^2) $ sized matrices involved in \eqref{e:small_psd} are considerably smaller than the $\bigO(n^2+d^2) $ matrix in \eqref{e:PMSDP_psd}. As a result the equivalent relaxation is considerably more efficient than \eqref{e:PM_SDP}. For an explanation of the equivalence see \cite{Haggai}.
\section{Properties of relaxation}\label{sec:properties}
We present some natural properties of the relaxation, which will be helpful for the proof of our main theorem. We begin by presenting two consequences from our discussion above.\vspace{3pt}
\begin{proposition}\label{p:first}
For $\mu \in \MF$, the objective \eqref{e:obj} satisfies $$\mu\brac{\fnorm{RP-QX}^2}\geq \fnorm{\mu[R] P-Q\mu[X]}^2$$
 In particular $\dConv$ is bounded from below by zero, and if $\dConv(P,Q)=0$ then
$$\mu[R]P=Q\mu[X] $$
\end{proposition}
\begin{proof}[Proof of proposition~\ref{p:first}]
We note that $\fnorm{RP-QX}^2 $ can be rewritten as $\sum_{i,j} p_{ij}^2$ for $p_{ij}=(RP-QX)_{ij}\in\PP_1(x)$. Using Proposition \ref{p:lasserre} we obtain

\begin{equation*}
\begin{aligned}
\mu\brac{\fnorm{RP-QX}^2}=\sum_{i,j}\mu\left(p_{ij}^2\right)\geq \sum_{i,j}\mu\left(p_{ij}\right)^2 = \fnorm{\mu[R] P-Q\mu[X]}^2 \geq 0
\end{aligned}
\end{equation*}
\end{proof}
\begin{proposition}\label{p:second}
$\mu[R], \mu[X]$ are in the convex hull of the orthogonal transformations and permutations, respectively.
\end{proposition}
\begin{proof}[Proof of proposition~\ref{p:second} ]
We begins by showing that $\mu[X]$ is in the convex hull of $\perm_n$, \ie, that $\mu[X]$ is doubly stochastic. The rows and columns of $\mu[X]$ are constrained to sum to one in \eqref{e:X1}, and each coordinate of $\mu[X]$ must be non-negative since by \eqref{e:diag}
$$\mu[X_{ij}] = \mu[X_{ij}^2]\geq 0$$
The convex hull of orthogonal matrices are matrices whose $2$-norm is not larger than one. For arbitrary  $v\in \RR^d$ using Proposition~\ref{p:lasserre} and \eqref{e:orth} we have
\begin{align*}
\norm{\mu[R]v}_2^2 \leq\mu\left[\norm{Rv}_2^2\right]=\mu\left[v^TR^TRv\right]=v^T\mu[R^TR]v=\norm{v}_2^2
\end{align*}
and therefore $\norm{\mu[R]}_2\leq 1$. \end{proof}

\subsection{Compactness} We show that $\MF$ is compact. To show that all coordinates of a semi-definite matrix are bounded, it is sufficient to show that its trace is bounded. Therefore for general quadratic relaxations it is sufficient to show that  $\sum_i \mu[x_i^2] $ is bounded. In our case $x$ consists of the coordinates of $R,X$, and so the compactness of $\MF$ follows from
\begin{align*}
\sum_{ij} \mu[X_{ij}^2]+\sum_{k,\ell}\mu[R_{k\ell}^2]=\sum_{ij} \mu[X_{ij}]+\sum_k\mu[(RR^T)_{kk}]
=n+d
\end{align*}

\subsection{Invariance to coordinate change and reordering} The Procrustes distance $\dist(P,Q)$ is invariant to orthogonal change of coordinates and reordering of the points, that is given $(R_0,X_0),(R_1,X_1)\in G$ we have
\begin{equation}\label{e:equivalence}
\dist(P,Q)=\dist(\hat{P},\hat{Q})
\end{equation}
where $\hat{P}=R_0 P X_0$ and $\hat{Q}= R_1 Q X_1$. We now show that our convex approximation $\dConv$  satisfies \eqref{e:equivalence} as well.

We define
$$(R_0,R_1,X_0,X_1)_* (\mu)=\hat{\mu} $$
where
$$\hat{\mu}[p(R,X)]=\mu[p(R_1RR_0^T,X_1^TXX_0)] $$
It can be verified that if $\mu \in \MF$ then $\hat{\mu}\in \MF$ as well. Also note that $(R_0,R_1,X_0,X_1)_*$ is a linear map whose inverse is $(R_0^T,R_1^T,X_0^T,X_1^T)_*$.
Finally, note that
$$\fnorm{R\hat{P}-\hat{Q}X}^2=\fnorm{(R_1^T R R_0) P - Q (X_1 X X_0^T)}^2$$ which implies that
$$\hat{\mu}\left[\fnorm{R\hat{P}-\hat{Q}X}^2 \right]=\mu \left[\fnorm{RP-QX}^2 \right] $$
It follows that \eqref{e:equivalence} holds, and there is a linear isomorphism taking the minimizers of PM-SDP$(P,Q)$ to the minimizers of PM-SDP$(\hat{P},\hat{Q})$.

\section{Exact recovery}\label{sec:exactRecovery}
In this section we prove our main theorem (Theorem~\ref{th:main}). We are given $P,Q$ which satisfy the conditions of the theorem. Our goal is to show that the set of exact orthogonal solutions $\Iso_R(P,Q)$ is equal to the extreme points of the set of convex exact orthogonal solutions $  \N_R(P,Q)  $. In fact it is sufficient to prove that the convex hull of $\Iso_R(P,Q)$ is equal to $\N_R(P,Q)$. This is because the extreme points of $\conv \Iso_R(P,Q)$ are precisely $\Iso_R(P,Q)$, which in turn follows from the fact that all orthogonal matrices are extreme points of $\conv(\OO(d))$.

 The inclusion $\conv \Iso_R(P,Q)\subseteq   \N_R(P,Q) $ follows immediately from the fact that each exact orthogonal solution is also a convex exact orthogonal solution.

The proof of the  opposite inclusion $  \N_R(P,Q) \subseteq \conv \Iso_R(P,Q) $ is more involved. The remainder of this section is devoted to proving this inclusion.

We first note that using the invariance of PM-SDP to applying permutations and orthogonal transformations we can assume w.l.o.g. that $P=Q$ since by assumption $\dist(P,Q)=0$. Additionally we can assume that $PP^T$ is a  diagonal matrix denoted by $D$. To see this we use the spectral decomposition
$$UPP^TU^{T}=D $$
and note that $\hat P=UP $ satisfies $\hat P \hat P^T=D $.

By multiplying the equation $RP=PX $ with its transpose it can be  seen that all exact solutions $(R,X) \in \Iso(P,P) $ satisfy $RD=DR$. In the simple spectrum case this implies that $R$ is diagonal and $R_{jj}\in \{-1,1 \} $. Thus the symmetry group of $P$ can be identified with a subgroup of $\{-1,1\}^d $. The proof of the theorem is based on the following generalization of these properties of exact solutions to exact convex solutions:
\begin{lemma}\label{l:diagonal}
Assume $\mu \in \N$. Then
\begin{enumerate}
\item $\mu\left[ \norm{RD-DR}_F^2 \right]=0$.
\item    $\mu \left[R_{ij}^2 \right]=\delta_{ij}$.
\item $\Rmu$ is diagonal, and $\Rmu_{jj}\in [-1,1] $.
\item If $\mu[X_{\ell j}]>0 $ then there is some $R \in \{-1,1 \}^d$ such that $RP_j=P_\ell $.
\end{enumerate}
\end{lemma}
\begin{proof}

 For given $\mu \in \N$ we have
\begin{align*}\mu\left[\fnorm{RD-DR}^2 \right]=\mu \left[\tr\left( DR^TRD \right) \right]
-2 \mu \left[\tr\left(RDR^TD\right)\right]+\mu \left[\tr\left(DRR^TD\right) \right]
\end{align*}

We show the expression above vanishes by showing that all three summands in the last expression are equal to $\fnorm{D}^2 $. For the first and third summands this follows immediately from the linearity of $\mu$ and the constraint \eqref{e:orth}. For the second summand, recall  that $D=PP^T $ and note that
\begin{align*}
\tr\left(RPP^TR^TPP^T \right)-\tr\left(PXX^TP^T PP^T \right)
 \in \langle RP-PX  \rangle_2
\end{align*}
Therefore
\begin{align*}
\mu \left[\tr\left(RDR^TD^T\right)\right]=\tr\left(P\mu[XX^T]P^T PP^T \right)=\fnorm{D}^2
\end{align*}
where for the last equality on the right we use the fact that
\begin{align*}
\mu[XX^T]=\sum_j \mu[X_jX_j^T]=\sum_j\diag \left(\mu[X_j]\right)=I_n
\end{align*}
This proves the first claim.
We now have that
$$0 =\mu \left[(RD-DR)_{ij}^2 \right]=(D_{jj}-D_{ii})^2 \mu[R_{ij}^2] $$
Since the diagonal elements of $D$ are distinct by the simple spectrum assumption, we see that for $i \neq j$ we have $\mu[R_{ij}^2]=0$ while when $i=j$ we have
$$\mu[R_{ii}^2]=\sum_j \mu[R_{ij}^2]=\mu[(RR^T)_{ii}]=1 $$
This proves the second claim.
Using Proposition~\ref{p:lasserre} it follows immediately that $\Rmu$ is diagonal, and $\Rmu_{jj}\in [-1,1] $.

We now prove our last claim. We  assume $\mu[X_{\ell,j}]>0 $, and our goal is to show that for any $i \in [d]$, $|P_{i \ell}|=|P_{i j}| $. We fix some $i$ and define a vector $v$ by
$$v_k=P_{ik}^2 $$
Our goal is to show that $v_{\ell}=v_j $. Note that for all $k$,

\begin{align}\label{e:all}
v_k=P_{ik}^2\mu[R_{ii}^2]
\overset{{(a)}}{=}\mu[(RP)_{ik}^2]\overset{(b)}{=}\mu[(PX)_{ik}^2]\overset{(c)}{=}v \mu[X_k] \nonumber
\end{align}
 where:

 $(a)$ follows from the fact that $(RP)_{ik}^2-P_{ik}^2R_{ii}^2$ is a member of $\langle R_{ij}| \; j \neq i \rangle_2$.

$(b)$ follows from the fact that
\begin{align*}
(RP)_{ik}^2-(PX)_{ik}^2=(RP-PX)_{ik}(RP+PX)_{ik}
 \in \langle (RP-PX)_{ik} \rangle_2
\end{align*}

and $(c)$ follows from
\begin{align*}
\mu[(PX)_{ik}^2]=\mu[P_{i \star}X_kX_k^T P_{i \star}^T]
=P_{i \star} \diag \left( \mu[X_k] \right)P_{i \star}^T
=v^T\mu[X_k]
\end{align*}

As \eqref{e:all} holds for all $k$, it follows that $v^T=v^T\mu[X] $. Since $\mu[X]$ is doubly stochastic it can represented as a convex combination of permutation matrices $X(s)$. Accordingly $v^T$ is a convex combination of the vectors $v^TX(s)$ whose norm is no larger than $||v||$, and this implies that $v^TX(s)=v^T$ for all $s$. There must be some $s$ for which $X_{\ell j}(s)=1$ and therefore
$$v_{j}=(v^TX(s))_j=v_{\ell} $$
\end{proof}

We now explain how the theorem is proved from the lemma. Let us assume w.l.o.g. that $P_1$ is the faithful column of $P$. Let $\J$ be the collection of indices $\ell$ such that $\mu[X_{\ell 1}]$ is strictly positive. By the last part of the lemma, for each $\ell \in \J$ there is some $R(\ell)\in\{-1,1\}^d $ such that $R(\ell)P_{1}=P_\ell$. Moreover due to faithfulness $R(\ell)$ is an exact orthogonal solution.

Now note that for all $i \in [d]$,
\begin{align*}
\mu[R_{ii}]P_{i1}=\mu[(RP)_{i1}]=\mu[(PX)_{i1}]=\left(\sum_{\ell \in \J}\mu[X_{\ell 1}]R_{ii}(\ell)\right)P_{i1}
\end{align*}
Note that if all coordinates of $P_1$ are non-zero then for each $i$ we can cancel out $P_{i1}$, so that $\mu[R] $ is a convex combination of exact orthogonal solutions as required:
\begin{equation}\label{e:conv_comb}
\mu[R]=\sum_{\ell \in \J} \mu[X_{\ell 1}]R(\ell)
\end{equation}

For the general case we define $\hat R(\ell)$ to be the diagonal matrix with diagonal elements
$$\hat R(\ell)_{ii}=\twopartdef{R(\ell)_{ii}}{P_{i1}\neq 0}{\mu[R]_{ii}}{P_{i1}=0}$$
We note that \eqref{e:conv_comb} holds when $R(\ell)$ is replaced with $\hat R(\ell)$. Thus it is sufficient to show that each $\hat R(\ell)$ is a convex combination of exact orthogonal solutions.

Since the diagonal elements of $\mu[R]$ are in $[-1,1]$, each $\hat R(\ell)$ can be written as a convex combination of matrices $R(\ell,\alpha) \in \{-1,1\}^d $ satisfying the condition
$$R(\ell,\alpha)_{ii}=\hat R(\ell)_{ii}=R(\ell)_{ii} \text{ if } P_{i 1} \neq 0$$
Note that $R(\ell,\alpha)P_1=R(\ell)P_1=P_{\ell} $ and so all matrices $R(\ell,\alpha)$ are exact orthogonal solutions due to faithfulness. This concludes the proof of the symmetric version of the theorem.

In the asymmetric case, what we showed is that PM-SDP returns the unique exact orthogonal solution of PM. It remains only to show that the $X $ coordinate obtained from PM-SDP agrees with the exact solution as well. This follows from the following lemma

\begin{lemma}\label{l:Xrecovery}
Assume $P,Q $ satisfy $\dist(P,Q)=0 $ and  $Q$ has no repeated points. Let $(R(0),X(0))\in G $ be an exact solution for PM. Then $X=X(0)$ is the unique doubly stochastic solution to the equation
\begin{equation}\label{e:Rknown}
R(0)P=QX
\end{equation}
        \end{lemma}
        \begin{proof}
                Note that since the columns of $Q$ are pairwise disjoint $X(0)$ is the unique solution of \eqref{e:Rknown}  over the set of permutation. It remains to show this is the case over the set of doubly stochastic matrices as well.

                Note that for any doubly stochastic $X$,
                \begin{align*}
                0\leq\fnorm{R(0)P-QX}^2
                =\fnorm{R(0)P}^2-2\tr{R(0)PX^TQ^T}+\fnorm{QX}^2
                \end{align*}
                and therefore
                \begin{align*}
                \tr{R(0)PX^TQ^T} \leq \frac{1}{2} \left( \fnorm{R(0)P}^2+\fnorm{QX}^2 \right)
                 \leq \frac{1}{2} \left( \fnorm{P}^2+\fnorm{Q}^2 \right)=\fnorm{P}^2
                \end{align*}
                 Therefore, the linear functional $\tr{R(0)PX^TQ^T}$ is maximized by $X=X(0)$ and for any other permutation it is suboptimal. It is therefore suboptimal also for any other matrix in the convex-hull of the permutations, namely the doubly-stochastic matrices.

        \end{proof}

\section{Finding extreme points}\label{sec:extreme}
Our goal in this section is to explain how the characterization of the convex hull of exact orthogonal solutions from Theorem~\ref{th:main} can be used to obtain \emph{all} exact solutions.

The main observation is that according to Theorem~\ref{th:main}, the convex hull of the exact orthogonal solutions is the intersection of the feasible set of the relaxation with the hyperplane defined by constraining the objective to be zero. Therefore, it is possible to preform convex optimization over this set. More specifically, we randomly draw a matrix $W\in \RR^{d^2}$ with a uniform distribution over the unit sphere (w.r.t. the Frobenius norm) and obtain exact orthogonal solutions by solving the following optimization problem:

\begin{subequations}\label{e:random}
        \begin{align}
        \max_\mu \quad & \tr \,{W\mu[R]}\\
        s.t.     \quad & \mu \in \MF \\
                 \quad & \mu[\fnorm{RP-QX}^2]=0 \label{e:obj_zero}
        \end{align}
\end{subequations}
We now show that the random algorithm returns a unique exact solution with probability one, and the probability of obtaining each exact solution is uniform.
\begin{proof}[Proof of Theorem~\ref{th:proj}]
        Let $(R(i),X(i)), i=0,\ldots,r-1 $ be the members of $\Iso(P,Q)$. We begin by considering the case $P=Q$. In this case $\Iso_R(P,P)$ is a group and we index the group so that $R(0)=I_d $.

        A maximizer $\mu$ for \eqref{e:random} satisfies $\mu[R]=R(i) $  iff $W $ is a member of the set
        $$\A_i=\{W| \,\tr (W^TR(i))>\tr(W^TR(j)), \forall j \neq i \} $$
        clearly the union of $\A_i $ is a disjoint union, and has probability one.
        We note that $R(\ell) \A_0=\A_{\ell} $. Additionally, the probability of the sets $\A_{\ell} $ is preserved under linear isometries of the space $\RR^{d \times d} $ endowed with the Frobenius inner product, and the map $W \mapsto R(\ell)W $ is such an isometry. Therefore the sets $\A_{\ell}=R(\ell)\A_0 $ all have the same probability.

        The general case where $P \neq Q$  follows from the fact that $R(0)^T\Iso_R(P,Q)=\Iso_R(P,P) $.

        To conclude the argument, note that by lemma~\ref{l:Xrecovery}  $\mu[R]=R(i)$ implies that $\mu[X]=X(i)$.
\end{proof}

\subsection{Stability}\label{sub:perturb}
In applications  PM  problems are usually contaminated with  a certain amount of noise. It is therefore important to verify that the solutions of near-exact problems can be recovered as well. To this end, we now consider the case of point clouds  $P(\delta),Q(\delta) $ which are obtained by perturbing point clouds $P(0)=P$ and $Q(0)=Q$ which fulfill the conditions of Theorem~\ref{th:main}.

We propose  a modified version of the algorithm described above to obtain all solutions of the unperturbed problem. Note that  generally $PM(P(\delta),Q(\delta)) $ won't  have exact convex solutions (solutions with zero optimal objective value), so that    \eqref{e:random} will become infeasible due to \eqref{e:obj_zero}. Therefore we relax this constraint and instead require that the objective of PM-SDP is smaller than some suitable $\epsilon(\delta)$. This results in the following optimization problem:
\begin{subequations}\label{e:random_modified}
        \begin{align}
        \max_\mu \quad & \tr{W\mu[R]}\\
        s.t.     \quad & \mu \in \MF \\
        \quad & \mu[\fnorm{RP(\delta)-Q(\delta) X}^2]\leq \epsilon(\delta) \label{e:obj_eps}
        \end{align}
\end{subequations}

An optimal choice for $\epsilon(\delta) $ is
$$\bar \epsilon(\delta)=\max_{i=0,\ldots,r-1} \fnorm{R(i)P(\delta)-Q(\delta)X(i)}^2  $$
 This choice assures that all members of  $\Iso(P,Q)$  are in the feasible set, while excluding as many irrelevant solutions as possible.

We claim that when $P(\delta),Q(\delta)$ are close to $P,Q$, we obtain close to exact solutions, with close to uniform probability. To make this statement more precise we need to introduce some notation.
We denote the minimizers of \eqref{e:random_modified} by $\mu_\delta(W)$, and their $R$ coordinate by $R_\delta(W)$. Note that both $\mu_\delta(W)$ and $R_\delta(W)$ are sets which may generally contain more than one solution. Also note that when $\delta=0$ the optimization problems \eqref{e:random_modified} and  \eqref{e:random} are identical, and so $R_0(W)$ is uniquely defined almost everywhere and attains all exact orthogonal solutions with uniform probability due to theorem~\ref{th:proj}. The solutions $R_\delta(W)$ are close to the unperturbed solution $R_0(W) $ if
$$r_\delta(W)=\inf \{r| \quad  R_\delta(W) \subseteq B_r(R_0(W)) \} $$
is small.  We would like to show that for any $\eta > 0$,
$$\mathrm{Prob}\{W| \, r_\delta(W)\geq \eta \} \overset{\delta \rightarrow 0}{\longrightarrow} 0 $$
In other words, we claim that $r_\delta$ converges in probability to zero.
We prove a slightly stronger claim
\begin{proposition}\label{th:noise}
        Assume $\epsilon(\delta)\geq \bar{\epsilon}(\delta) $ and $\epsilon(\delta)\overset{\delta \rightarrow 0}{\longrightarrow} 0 $. Then
        $$r_\delta \overset{a.s.}{\longrightarrow} 0 $$
\end{proposition}
In Appendix~\ref{a:meas} we show that $r_\delta$ is Borel measurable and so the probabilistic language used above is justified. The proposition above is essentially a known result on perturbations of SDPs (see \cite{bonnans2013perturbation} pages 492-493). We include the proof for completeness.
\begin{proof}
We show pointwise convergence for every $W$ such that $R_0(W) $ is uniquely defined. Assume by contradiction that there is a sequence $\delta_n \longrightarrow 0$ such that $r_{\delta_n}(W) \longrightarrow a>0 $. We can then choose  $\mu_n \in \mu_{\delta_n}(W) $ such that
\begin{equation}
\label{e:contradiction}
|\mu_n[R]-R_0(W)|\longrightarrow a
\end{equation}
By moving to a subsequence, we can assume that $\mu_n $ converges to some $\mu$. We obtain a contradiction to \eqref{e:contradiction} by showing that $\mu[R]=R_0(W) $. This is because
$$
\tr{\,WR_0(W)}\overset{(*)}{\leq} \tr{\,W\mu_n[R]}\longrightarrow \tr{\,W\mu[R]} \overset{(**)}{\leq} \tr{\,WR_0(W)}
$$
where $(*)$ follows from the fact that by assumption $R_0(W) $ is a feasible point of \eqref{e:random_modified}, and $(**)$ follows from the fact that $\mu$ is a feasible point of \eqref{e:random}  since $\epsilon(\delta) \longrightarrow 0 $. It follows that $\mu$ is a maximizer of \eqref{e:random} and therefore $\mu[R]=R_0(W)$ due to the uniqueness of $R_0(W)$.
\end{proof}

\section{Experiments}\label{sec:exp}
In this section we present experiments illustrating our theoretical results, and experimentally explore the exactness of PM-SDP in cases not covered by our analysis, \ie, when either the simple spectrum condition or the  faithfulness condition fail.

To illustrate our results for the asymmetric case we conducted the following experiment. A point cloud $P$ was generated from an i.i.d. normal distribution, and then centralized. $Q$ was obtained by applying random $R,X$ to $P$. We then solved PM-SDP on the generated point cloud. This process was repeated one-thousand times and in all experiments the unique exact solution with objective value $0$ was obtained as can be seen in Figure~\ref{fig:hist}(a).

To illustrate our results for the symmetric case we conducted an experiment which we have already mentioned earlier (Figure~\ref{fig:sym}). We solve PM-SDP on planar point clouds $P=Q$ whose symmetries are
\begin{equation*}
R_0=I_2, \quad R_1=\left( \begin{array}{cc}
-1 & 0   \\
0 & 1
\end{array}  \right)
\end{equation*}
As predicted, the $R$ coordinate of PM-SDP is in the convex hull of $R_i, i=0,1$. We then ran one-hundred iterations of the random algorithm described in Section~\ref{sec:extreme}. As predicted, in all iterations one of the solutions $R_i$ was attained (up to a small numerical error). We also proved that $R_0$ and $R_1$ should be attained with uniform probability. In this case $R_0$ was attained in $51$ experiments and $R_1$ in $49$ experiments.

We now discuss the cases which are not covered by our analysis.
 We begin by applying PM-SDP  to the asymmetric unfaithful example from Section~\ref{sec:intro}. The results are shown on the top row of Figure~\ref{fig:tightness}. We solve PM-SDP, taking $P$ (green) to be the unfaithful shape, and $Q$ (blue) to be a relabeled, rotated version of $P$. It turns out that although $P$ doesn't meet the conditions of Theorem~\ref{th:main} due to unfaithfulness,  PM-SDP is still able to recover the correct orthogonal transformation and permutation. Indeed, as can be seen on the right, applying the obtained transformation  to $P$ we are able to align the two shapes perfectly.

In contrast, we provide an example for a shape with multiple eigenvalues where exact recovery is not obtained. These results are shown in the bottom row of Figure~\ref{fig:tightness}. The input shape $P$ is asymmetric, but applying PM-SDP to $P$ (green) and  $Q$ (blue) chosen to be a rotated, relabeled version of $P$, does not yield a solution in $G$. To obtain a solution in $G$ we used the projection procedure described in \cite{Haggai}. The obtained solution is still not the correct exact solution as shown in the image on the bottom right which is not aligned correctly.

We also conducted some random experiments. The results are summarized in Figure~\ref{fig:hist}. Each experiment was run one-thousand times on point clouds $P$ and $Q$ obtained from $P$ by applying a random orthogonal transformation and permutation.

Experiment (b) quantitatively shows that exact recovery is obtained for unfaithful shapes as well. Each $P$ here is generated by superimposing two shapes with different reflectional symmetry as in the unfaithful example in Figure~\ref{fig:tightness}.

Experiment (c) quantitatively shows that exact recovery isn't always obtained for shapes with eigenvalue multiplicity. Here the two dimensional asymmetric shapes $P$ are generated by superimposing two shapes with non-bilateral symmetry groups. In all cases PM-SDP returns a solution whose $R,X$ coordinates aren't in $G$, and we project onto $G$ using the scheme described in \cite{Haggai}. It can be seen that in some cases the PM objective obtained from $R,X$ after projection is non-zero although an exact solution does exist.

Experiment (d) examines a different class  of shapes with eigenvalue multiplicity for which exact recovery is obtained. Here $P$ is generated from an i.i.d. normal distribution. It is then centralized and scaled along one of its principal axes so that a shape with eigenvalue multiplicity is obtained.
In all cases PM-SDP returned the correct exact solution of the problem (even before projection).


\begin{figure}[t]
        \centering
        \includegraphics[width=0.7\columnwidth]{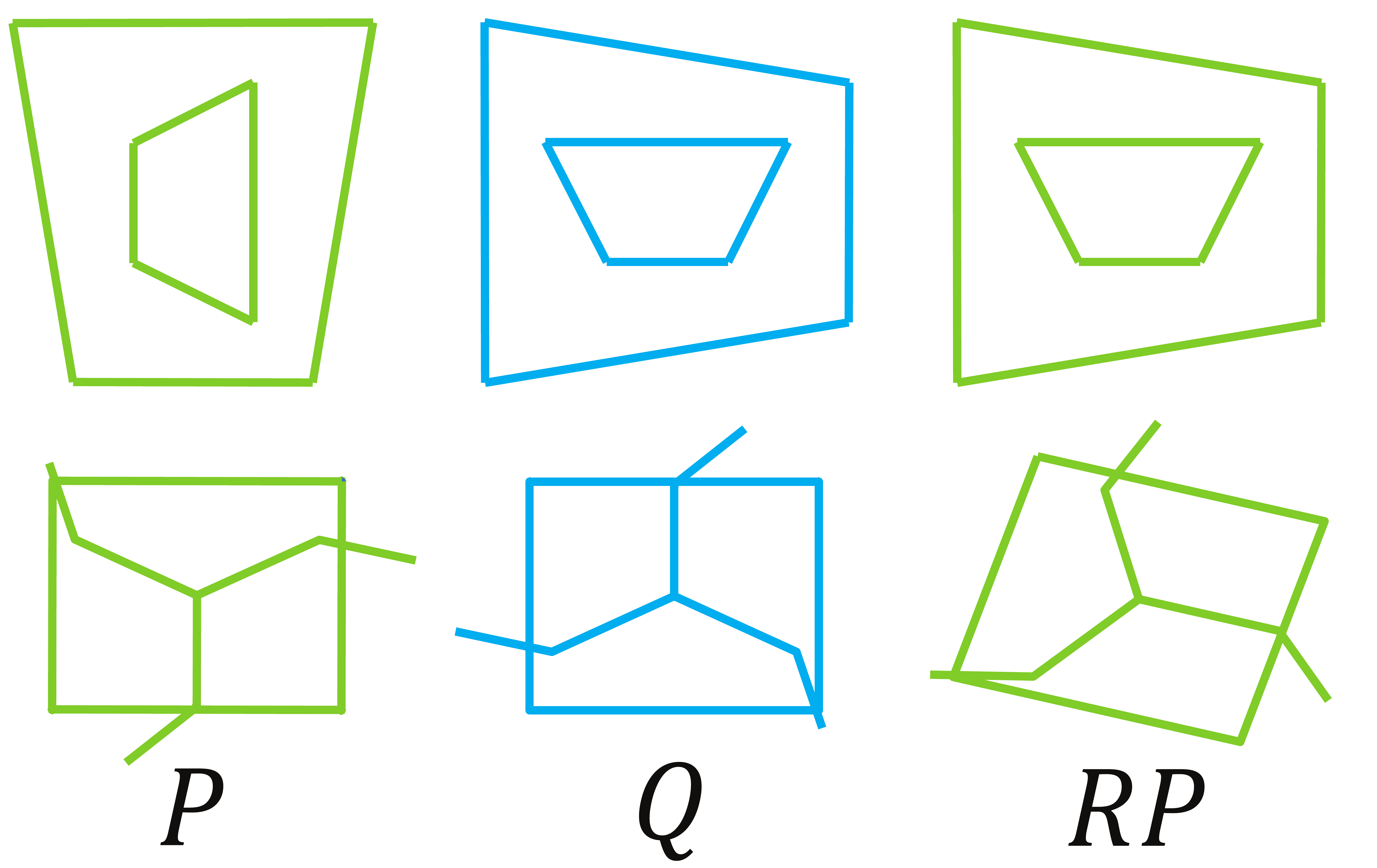}
        \vspace{-0cm}
        \caption{Exact recovery when the conditions of Theorem~\ref{th:main} aren't met. By applying PM-SDP to the unfaithful point clouds $P$ (green, top left) and $Q$ (top blue) the correct orthogonal transformation and permutation taking $P$ to $Q$ can be reconstructed ($RP$, green, top right). The bottom row shows that PM-SDP fails to align the shapes $P$ and $Q$ which violate the simple spectrum condition.}
        \vspace{-0.3 cm}
        \label{fig:tightness}
\end{figure}

 To summarize, we are not aware of simple spectrum, unfaithful shapes, for which PM-SDP does not achieve exact recovery. Thus a possible direction for future work is investigating whether this condition can be removed or weakened. In contrast we provide examples of  shapes with double eigenvalue multiplicity for which exact recovery fails, and thus the simple spectrum condition is tight.
\section{Exact graph matching}\label{sec:complexity}
The purpose of this section is establishing that PM, and in fact even the subproblem of exact PM, is computationally hard. We do this by showing equivalence of exact PM and exact graph matching.

For our discussion here we define the exact PM problem as the problem of deciding whether $\dist(P,Q)=0$ for given point clouds $P,Q \in \RR^{d \times n} $.
The input of the exact graph matching (GM) problem is a pair of graphs $G_A=(V_A,A) $ and $G_B=(V_B,B) $. The vertices of the graphs are assumed to be of equal cardinality $n$. The matrices $A,B$  are assumed to be symmetric and are chosen to reflect relevant information on the graph. The goal of the exact graph matching problem is deciding whether there is a relabeling of the vertices so that the graphs are identical. Mathematically this amounts to deciding whether there is a permutation matrix $X \in \perm_n$ such that
\begin{equation}
\label{e:GM}
XA=BX
\end{equation}

In the special case where $A,B$ are the adjacency matrices of the graphs, this problem is called the graph isomorphism problem. The graph isomorphism problem is in NP, but is not known to be either NP-complete or in P. We now prove the equivalence of exact PM and exact GM:
\begin{proof}[proof of Theorem~\ref{th:GM}]
The main observation needed for the proof is that if $P,Q \in \RR^{d \times n}$ satisfy $P^TP=A $ and $Q^TQ=B $, then the set of  solutions of \eqref{e:GM} is exactly the projection onto the $X$ coordinate of the set of solutions to the exact PM equation
\begin{equation}\label{e:PMhelper}
RP=QX
\end{equation}
To see this note that if \eqref{e:PMhelper} holds then
\begin{equation*}
XA=X(RP)^TRP=X(QX)^TQX=BX
\end{equation*}
On the other hand, if \eqref{e:GM} holds then
$$P^TP=X^TQ^TQX $$
and so $P$ and $QX $ are two factorizations of the same matrix. Since factorization is unique up to an orthogonal transformation there must be some $R \in \OO(d) $ such that \eqref{e:PMhelper} holds.

\begin{figure}[t]
        \centering
        \begin{tabular}{c c c c}
                \includegraphics[width=.2\columnwidth]{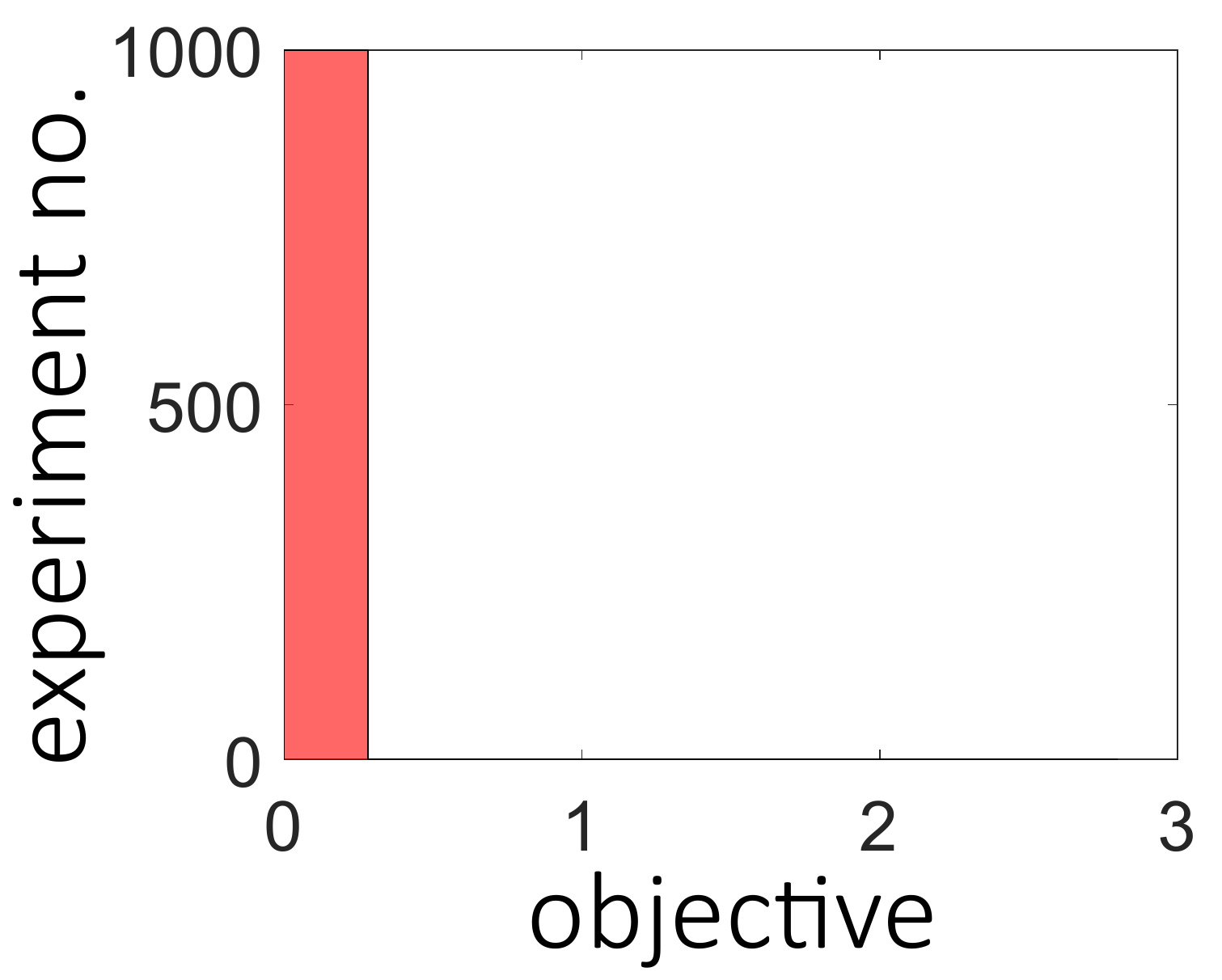}&
        \includegraphics[width=.2\columnwidth]{Hist_a_Fat.pdf}&
        \includegraphics[width=.2\columnwidth]{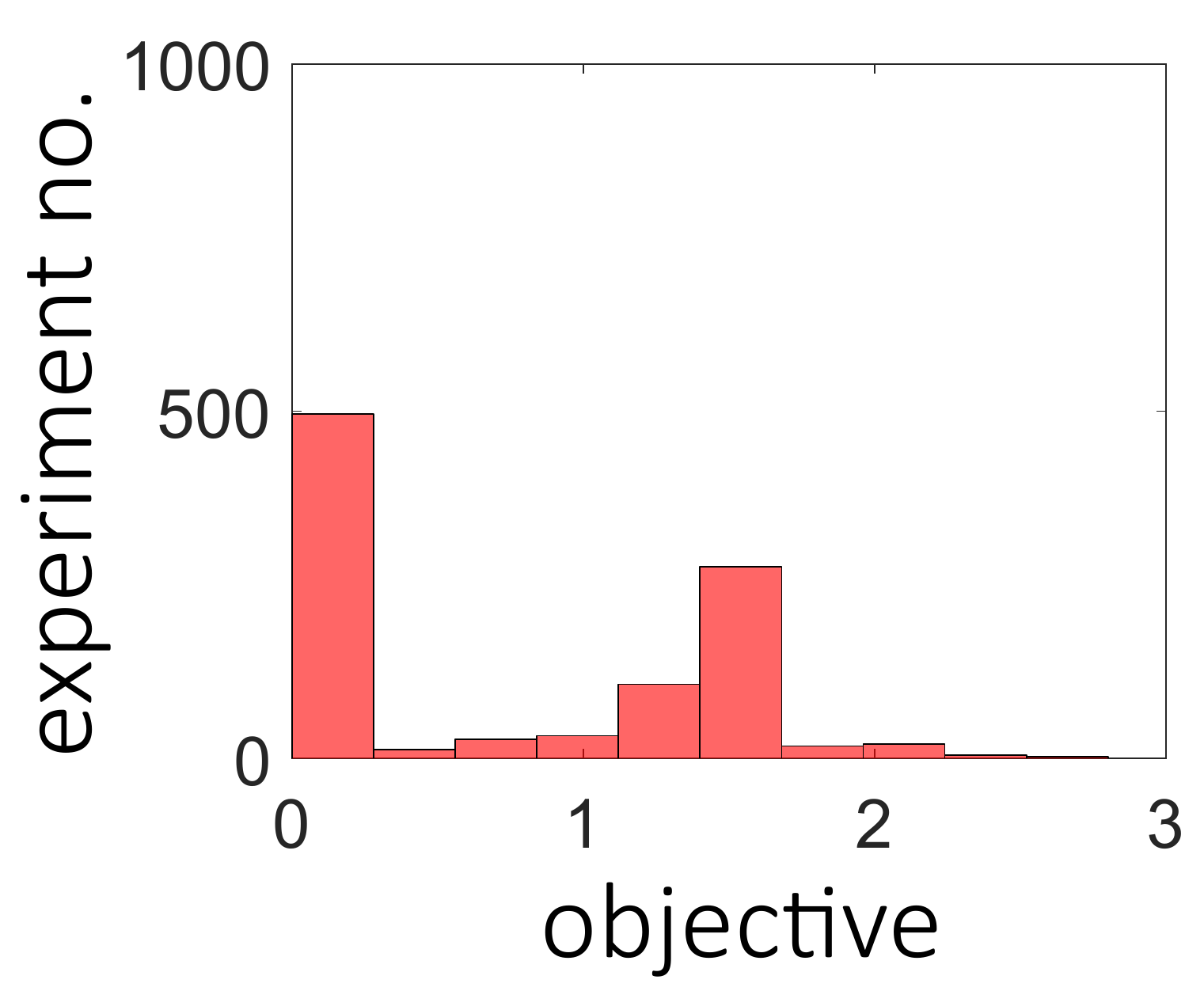}&
        \includegraphics[width=.2\columnwidth]{Hist_a_Fat.pdf}\\
                \quad \, (a)& \quad \, (b)& \quad \,  (c) & \quad \, (d) \\

        \end{tabular}

        \vspace{-0.3cm}
        \caption{Exact recovery experimental evaluation. When the conditions of the (asymmetric) exact recovery theorem are met, an optimal solution with objective zero is always attained (a). Exact recovery is also obtained for randomly generated unfaithful point clouds (b). For randomly generated shapes with eigenvalue multiplicity exact recovery can fail (c). In (d) successful exact recovery is shown for an easier class of shapes with eigenvalue multiplicity.}
        \vspace{-0.3 cm}
        \label{fig:hist}
\end{figure}

We now prove the theorem using this observation. We begin by showing a polynomial reduction from exact GM to exact PM. Assume we are given two graphs defined by the matrices $A,B $. We note that we can assume w.l.o.g. that $A,B \succeq 0$ since for any $\lambda \in \RR$, a permutation $X$ satisfies \eqref{e:GM} if and only if
$$X(A+\lambda I_n)=(B+\lambda I_n)X $$
We can therefore factorize $A,B $ to obtain $P,Q $, and the observation implies that the graphs are isomorphic iff $d(P,Q)=0 $.

For the opposite reduction- assume we are given point clouds $P,Q $. We define $A=P^TP $ and $B=Q^TQ$. Our observation implies that $d(P,Q)=0$ if and only if there is a solution to \eqref{e:GM}.
\end{proof}

\appendix
\section{Measurability}\label{a:meas}
We show $r_\delta $ is upper semi-continuous at almost every $W$ and hence Borel measurable. Assume by contradiction that there is some sequence $W_n \longrightarrow W $ such that $ r_\delta(W_n)\longrightarrow a>r_\delta(W) $. We can then choose  $R_n \in R_\delta(W_n) $ such that
\begin{equation}
        \label{e:contradiction2}
        |R_n-R_0(W)| \longrightarrow a
\end{equation}
by moving to a subsequence we can assume that $R_n$ converges to some $R$, using similar reasoning as in the proof of Theorem~\ref{th:noise} it can be shown that $R \in R_\delta(W) $. It follows that $|R-R_0(W)| \leq r_\delta(W)$ in contradiction to \eqref{e:contradiction2}.

\section*{Acknowledgements}
We would like to thank Haggai Maron for providing Figure~\ref{fig:hm}. This research was supported in part by the European Research Council (ERC Starting Grant ``SurfComp'', Grant No.~307754), the Israel Science Foundation (Grant No.~1284/12), I-CORE program of the Israel PBC and ISF (Grant No.4/11).

\bibliographystyle{siamplain}
\bibliography{procrustesPNAS}

\end{document}